\documentclass[final,3p,times,11pt]{elsarticle}
\usepackage[USenglish]{babel}
\usepackage{amsmath,amssymb,amsthm, mathrsfs,multirow}
\usepackage{mathtools}
\usepackage{graphicx}
\usepackage{stmaryrd}
\usepackage[dvipsnames]{xcolor}
\usepackage{cancel}
\usepackage{ulem}
\usepackage{tabularx}
\usepackage{comment}
\usepackage[linesnumbered,ruled,vlined]{algorithm2e}
\definecolor{Myblue}{rgb}{.2 0.4 1}

\usepackage{hyperref}
\hypersetup{
    colorlinks = true,       
}

\newtheorem{theorem}{Theorem}

\journal{}
\makeatletter
\def\ps@pprintTitle{%
 \let\@oddhead\@empty
 \let\@evenhead\@empty
 \def\@oddfoot{}%
 \let\@evenfoot\@oddfoot}
\makeatother

\begin{document}
\begin{frontmatter}
\title{Recursive rounding of sample size estimation for multi-fidelity Monte Carlo}

\author[RiceCMOR]{Jiaxing Liang}
\ead{jliang18@terpmail.umd.edu}
\begin{abstract}
In multifidelity Monte Carlo (MFMC), optimal sample allocations are typically derived from a continuous relaxation of a variance minimization problem, with integer solutions obtained through post hoc rounding. Such rounding procedures may fail to fully exploit the available cost budget, particularly under tight cost budgets or when model costs vary significantly. In this work, we reformulate the MFMC allocation problem as a variance-constrained cost minimization problem that is equivalent to the standard budget-constrained formulation at the continuous level. 
This reformulation admits a recursive structure that enables the construction of an integer allocation strategy based on Bellman's principle of optimality. Since the resulting MFMC allocation has a mathematical structure similar to the optimal multilevel Monte Carlo (MLMC) allocation, the proposed strategy naturally extends to MLMC. The resulting algorithm constructs integer-valued sample allocations that more closely follow the continuous variance--cost tradeoff while using the prescribed variance tolerance more efficiently. Numerical experiments demonstrate that the proposed approach  satisfies the prescribed variance tolerance with less computational overhead than standard rounding strategies.
\end{abstract}

\end{frontmatter}

\section{Introduction}\label{sec:Intro}

Monte Carlo (MC) methods provide a standard framework for the statistical characterization of complex systems in science and engineering. Their practical use, however, is often limited by the high computational cost of high-fidelity (HF) model evaluations. Multifidelity Monte Carlo (MFMC) methods \cite{PeWiGu:2016} address this challenge by introducing a hierarchy of low-fidelity (LF) surrogate models that are computationally inexpensive yet sufficiently correlated with the HF quantity of interest. By combining these models appropriately, MFMC constructs an unbiased estimator with substantially reduced variance compared to standard MC sampling.

A central component of MFMC is the determination of optimal sample allocations across the model hierarchy. This problem is commonly formulated in two equivalent ways: the {\it cost-constrained variance minimization problem} and the {\it variance-constrained cost minimization problem}. In the continuous setting, these formulations are analytically equivalent; however, when restricted to integer-valued sample sizes, their behavior differs. The traditional budget-constrained approach \cite{PeWiGu:2016}, which uses flooring for large budgets and hybrid rounding \cite{GrGuJuWa:2023} to avoid inadmissible zero-sample assignments in low-budget regimes, often leads to inflated estimator variance. This occurs because post hoc rounding leaves part of the computational budget unused, an effect that is more pronounced when resources are limited or model costs differ significantly. In contrast, the variance-constrained formulation naturally ensures admissibility by enforcing at least one sample per model, but typically relies on ceiling operations to meet the prescribed error tolerance. As a result, it incurs additional computational cost by oversatisfying the variance constraint. In both cases, the resulting integer allocation deviates from the continuous cost--variance Pareto front and leads to inefficient use of resources.

To address this issue, we propose an iterative redistribution scheme for the variance-constrained cost minimization problem. The method reallocates the remaining variance tolerance across unresolved fidelity levels, thereby reducing the gap between the integer solution and the continuous optimum. In contrast to one-step rounding procedures, this approach redistributes the integer slack and yields allocations that more closely follow the optimal variance--cost tradeoff.

The proposed approach has two key features. First, we reformulate the MFMC allocation problem as a variance-constrained cost minimization problem. This reformulation preserves the Pareto-optimal structure of the equivalent budget-constrained formulation while guaranteeing admissibility by ensuring a nontrivial highest-fidelity sample size, which eliminates the need for heuristic corrections in low-budget regimes. Second, we observe that the resulting allocation has a similar structure as the optimal multilevel Monte Carlo (MLMC) allocation \cite{Gi:2008,Gi:2015}. Exploiting this structural correspondence, we develop a recursive rounding strategy that constructs integer-valued sample allocations by sequentially updating the residual variance budget after each rounding step. The resulting framework applies to both MFMC and MLMC.

The remainder of this paper is organized as follows. In Section~\ref{sec:MFMC}, we review the MFMC framework from a variance-constrained perspective and discuss the governing continuous allocation theory. Section~\ref{sec:Iterative_IntegerValued_Sample_Size} introduces the proposed iterative integer-valued allocation algorithm and establishes its fundamental properties. Section~\ref{sec:Num_Result} provides numerical evidence validating the performance of our approach.

\section{Multi-fidelity Monte Carlo}\label{sec:MFMC}
Let $(\Omega, \mathcal{F}, \mathbb{P})$ be a complete probability space, where $\Omega \subset \mathbb{R}^d$ is the space of elementary events, $\mathcal{F}$ is the associated $\sigma$-algebra, and $\mathbb{P}:\mathcal{F} \to [0,1]$ is a probability measure. Let $(\mathcal{U}, \langle \cdot, \cdot \rangle_{\mathcal{U}})$ be a separable Hilbert space. The model outputs are represented as random variables in the Bochner space
\begin{equation}\label{eq:BochnerL2}
    L^2_{\mathbb{P}}(\Omega; \mathcal{U}) := \left\{ u: \Omega \to \mathcal{U} \;\middle|\; u \text{ is strongly measurable and } \int_\Omega \|u(\boldsymbol{\omega})\|_{\mathcal{U}}^2 \, \mathrm{d}\mathbb{P}(\boldsymbol{\omega}) < \infty \right\}.
\end{equation}
The space is equipped with the inner product $\langle u, v \rangle =  \int_\Omega  \langle  u(\boldsymbol{\omega}) , v(\boldsymbol{\omega})  \rangle_{{\mathcal U}} \; \textup{d}\mathbb{P}(\boldsymbol{\omega})$.

Let $u_1 \in L^2_{\mathbb{P}}(\Omega;\mathcal{U})$ denote the high-fidelity (HF) model, whose expectation $\mathbb{E}[u_1]$ is the quantity of interest. Direct Monte Carlo estimation of this expectation is often computationally expensive. Multifidelity Monte Carlo (MFMC) alleviates this cost by combining the HF model with a hierarchy of increasingly inexpensive low-fidelity (LF) approximations $\{u_k\}_{k=2}^K$. For each model $u_k$, define the standard Monte Carlo estimator
\[
A^{\mathrm{MC}}_{k,N}
:=
\frac{1}{N}\sum_{i=1}^N u_k(\omega^{(i)}),
\]
based on $N$ independent and identically distributed (i.i.d.) samples. The MFMC estimator combines HF and LF estimators through a control-variate construction,
\begin{equation}\label{eq:MFMC_estimator}
    A^{\mathrm{MF}}
    = A^{\mathrm{MC}}_{1,N_1}
    + \sum_{k=2}^K \alpha_k\left( \overline{A}_{k,N_k} - \overline{A}_{k,N_{k-1}} \right),
\end{equation}
where $\alpha_k$ are the control-variate weights and $N_k$ are the corresponding sample sizes. We impose a nested sampling strategy satisfying $N_1 \le N_2 \le \dots \le N_K$. This strategy ensures that the two estimators in each correction term, $\overline{A}_{k,N_k}$ and $\overline{A}_{k,N_{k-1}}$, share common random samples, exploiting the correlation between estimators at the same fidelity level and reducing the variance of the correction term.

The nested sampling strategy couples the correction terms. An equivalent decomposition into uncorrelated blocks is obtained by defining
\begin{equation}\label{eq:MFMC_Yk}
Y_1 := A^{\text{MC}}_{1,N_1},\quad 
Y_k := \left(1-\frac{N_{k-1}}{N_k}\right)\!\left(A^{\text{MC}}_{k,N_k\backslash N_{k-1}} - A^{\text{MC}}_{k,N_{k-1}}\right), \;\; k=2\ldots, K,
\end{equation}
where $A^{\mathrm{MC}}_{k,N_k \setminus N_{k-1}}$ denotes the MC estimator constructed from the additional
$N_k-N_{k-1}$ samples, which are independent of the
$N_{k-1}$ samples used in $A^{\mathrm{MC}}_{k,N_{k-1}}$.
With these definitions, the MFMC estimator can be written as
\begin{equation}\label{eq:MFMC_estimator_independent}
    A^{\mathrm{MF}} = Y_1 + \sum_{k=2}^{K} \alpha_k Y_k
\end{equation}
By construction, $\mathbb{E}[Y_k]=0$ for $k\ge2$, so the estimator remains unbiased. Although the correction terms are constructed from nested samples, the corresponding variables
$Y_2,\ldots,Y_K$
are pairwise uncorrelated, i.e.,
\[
\operatorname{Cov}(Y_k,Y_j)=0,\qquad
2\le k<j\le K.
\]
This property follows from the nested sampling structure: after expanding the covariance between two correction terms, the contributions from disjoint sample increments vanish by independence, while those from the shared samples cancel exactly. Consequently, the only nonzero covariance terms are those between $Y_1$ and $Y_k$, $k\ge2$. Denote the variance of the $k$-th model and the Pearson correlation coefficients between the $k$-th and $j$-th models by
\begin{equation*}
    \sigma_k^2 = \mathbb{V}\!\left[u_k(\boldsymbol{\omega})\right],\qquad 
\rho_{k,j}=\frac{\operatorname{Cov}\!\left[u_k(\boldsymbol{\omega}),u_j(\boldsymbol{\omega})\right]}{\sigma_k\sigma_j}, 
    \quad k,j=1,\dots,K,
\end{equation*}
where the covariance is defined as $\text{Cov}[u_k(\boldsymbol{\omega}),u_j(\boldsymbol{\omega})] := \mathbb{E}[\langle u_k(\boldsymbol{\omega}) - \mathbb{E}[u_k(\boldsymbol{\omega})], u_j(\boldsymbol{\omega}) - \mathbb{E}[u_j(\boldsymbol{\omega})]\rangle_\mathcal{U}]$, and $\rho_{k,k}=1$. It follows from \cite[Lemma~3.2]{PeWiGu:2016} that the variance and covariance of $Y_k$ are 
\begin{equation*}
\label{eq:Var_Cov_Yk}
\mathbb{V}[Y_1] = \frac{\sigma_1^2}{N_1}, \quad
\mathbb{V}[Y_k] = \left(\frac{1}{N_{k-1}} - \frac{1}{N_k}\right)\sigma_k^2,\quad
\operatorname{Cov}(Y_1, Y_k)
= -\left(\frac{1}{N_{k-1}} - \frac{1}{N_k}\right)\rho_{1,k}\sigma_1\sigma_k,\quad k=2,\ldots,K.
\end{equation*}
Consequently, the variance of the MFMC estimator is
\begin{equation}
\label{eq:MFMC_variance}
\mathcal{V}^{\mathrm{MF}} = \frac{\sigma_1^2}{N_1} + \sum_{k=2}^{K} \left( \frac{1}{N_{k-1}} - \frac{1}{N_k} \right) \left( \alpha_k^2 \sigma_k^2 - 2 \alpha_k \rho_{1,k} \sigma_1 \sigma_k \right).
\end{equation}
Let $C_k$ denote the cost of a single evaluation of model $u_k$. The total cost to compute the MFMC estimator is

\[
\mathcal{W}^{\mathrm{MF}} = \sum_{k=1}^K C_k N_k.
\]

\subsection{Continuous MFMC allocation}

The objective is to determine the optimal weights $\alpha_k$ and sample sizes $N_k$. For compactness, let
\[
\boldsymbol{\alpha}=(\alpha_1,\ldots,\alpha_K),
\qquad
\boldsymbol{N}=(N_1,\ldots,N_K),
\]
with $\alpha_1=1$ corresponding to the high-fidelity model. Rather than adopting the budget-constrained formulation of \cite{PeWiGu:2016}, we consider the equivalent formulation of minimizing computational cost subject to a prescribed variance tolerance $\epsilon^2$. The two formulations share the same Pareto-optimal frontier and therefore yield the same continuous optimum. Accordingly, we consider the optimization problem
\begin{equation}\label{eq:Optimization_pb_sample_size}
    \begin{array}{ll}
    \displaystyle\min_{\boldsymbol{\alpha} \in \mathbb{R}^{K}, \boldsymbol{N} \in \mathbb{R}^{K}} &\displaystyle \mathcal{W}^{\text{MF}}(\boldsymbol{N}),\\
       \text{subject to} &\mathcal{V}^{\text{MF}}(\boldsymbol{\alpha},\boldsymbol{N})=\epsilon^2,\\[2pt]
       &\displaystyle N_1\ge 0,\quad \displaystyle N_{k-1}\le N_k, \;\; k=2\ldots,K.
    \end{array}
\end{equation}
The analytical solution is stated in Theorem~\ref{thm:Sample_size_est}, with the derivation deferred to the Appendix.

\begin{theorem}
\label{thm:Sample_size_est}
Consider the optimization problem \eqref{eq:Optimization_pb_sample_size} for
$K$ models $\{u_k\}_{k=1}^K$
with standard deviations $\sigma_k$, correlation coefficients
$\rho_{1,k}$ relative to the high-fidelity model $u_1$, and evaluation costs
$C_k$. Define $\Delta_k = \rho_{1,k}^2 - \rho_{1,k+1}^2$ for $k = 1, \dots, K$ with $\rho_{1,K+1}=0$. Assume the following conditions hold
\begin{alignat*}{5}
&(i)\;\textit{Monotone correlations:} &\quad\quad\quad& |\rho_{1,2}| > \cdots > |\rho_{1,K}|>0,\\
&(ii)\;\textit{Cost-correlation ratio:} &\quad& \frac{\Delta_{k}}{C_k} > \frac{\Delta_{k-1}}{C_{k-1}}, \quad\;\; k=2,\ldots,K.
\end{alignat*}
Then the optimal weights $\alpha_k^*$ and sample sizes $N_k^*$ are 
\begin{equation}\label{eq:MFMC_RealValued_Sample_Size}
    \alpha_k^* = \frac{\rho_{1,k}\sigma_1}{\sigma_k}, \qquad
    \;N_k^*=\frac{\sigma_1^2}{\epsilon^2}\sqrt{\frac{\Delta_k}{C_k}}\sum_{j=1}^K\sqrt{C_j\Delta_{j}}, \quad\;\; k=1,\ldots,K.
\end{equation}
%
%
\end{theorem}
\vspace{4mm}
At the continuous optimum \eqref{eq:MFMC_RealValued_Sample_Size}, the total computational cost and the {\it normalized variance} are
\begin{equation}\label{eq:MFMC_sampling_cost}
    \mathcal{W}^\text{MF} = \sum_{k=1}^K C_k N_k^* = \frac{\sigma_1^2}{\epsilon^2}\left(\sum_{k=1}^K\sqrt{C_k\Delta_k}\right)^2,
    \quad \quad
    \mathcal{V}^{\mathrm{MF}}_n :=\frac{\mathcal{V}^{\mathrm{MF}}}{\sigma_1^2} = \sum_{k=1}^K \frac{\Delta_k}{N_k}.
\end{equation}
The normalization yields a dimensionless variance measure, and the prescribed tolerance becomes
$\mathcal{V}^{\mathrm{MF}}_n=\epsilon^2/\sigma_1^2$. The computational efficiency of MFMC relative to standard Monte Carlo is quantified by the cost ratio 
\begin{equation}\label{eq:MFMC_sampling_cost_efficiency}
    \xi = \frac{\mathcal{W}^\mathrm{MF}}{\mathcal{W}^\mathrm{MC}} 
    = \frac{1}{C_1}\left(\sum_{k=1}^K \sqrt{C_k \Delta_k}\right)^2,
\end{equation}
where $\xi<1$ indicates that MFMC achieves the prescribed variance at lower cost than standard Monte Carlo.

Differentiating $\mathcal{V}^{\mathrm{MF}}_n$ and $\mathcal{W}^{\mathrm{MF}}$ with respect to $N_k$ reveals the local variance–cost trade-off underlying the MFMC allocation: increasing $N_k$ reduces variance but increases computational cost. At the continuous optimum, the marginal variance reduction per unit computational cost,  $\frac{\Delta_k}{C_k N_k^2}$, is identical across all active fidelity levels. This characterizes an \textit{equal-marginal-efficiency principle}: every model contributes the same marginal variance reduction per unit cost. Consequently, the MFMC allocation equalizes marginal efficiencies across fidelities and lies on the Pareto-optimal cost--variance frontier.

An additional observation is that the optimal sample allocation for MFMC shares essentially the same mathematical structure as that of multilevel Monte Carlo (MLMC) \cite{Gi:2008,Gi:2015}. Although the two estimators are constructed differently, their continuous optimal sample allocations differ only through the definition of the variance contribution associated with each level: $\sigma_1\Delta_k$ in MFMC plays an analogous role to the level-correction variance in MLMC.

\subsection{Model selection}

In practice, the per-sample costs $C_k$ and correlations $\rho_{1,k}$ associated with a collection of candidate models are unknown and must be estimated from pilot simulations. Given these estimates, including all available fidelity levels is not necessarily optimal, since some models may not contribute to the Pareto-optimal cost--variance tradeoff. A model-selection procedure is therefore required to identify a subset of fidelity levels that satisfies the assumptions of Theorem~\ref{thm:Sample_size_est} and minimizes the continuous MFMC cost \eqref{eq:MFMC_sampling_cost}. Let $\widetilde{\mathcal{S}}=\{1,\dots,\widetilde K\}$ denote the full set of available models. We seek an ordered subset
$\mathcal{S}=\{i_1,\dots,i_K\}\subseteq\widetilde{\mathcal{S}}$,
where $1=i_1<i_2<\cdots<i_K\leq \widetilde K$ ensures that the high-fidelity model is always included. Although the optimal subset may generally depend on the prescribed tolerance $\epsilon$, especially in the large-tolerance regime, we assume here that the selection depends only on the estimated costs and correlation information. Under this assumption, the same subset $\mathcal{S}$ can be reused for different prescribed variance tolerances.

The approach in \cite{PeWiGu:2016} uses exhaustive search to identify subsets satisfying the monotonicity and cost-correlation ordering conditions in Theorem~\ref{thm:Sample_size_est} while minimizing the resulting MFMC cost. To improve the search efficiency, we instead consider a depth-first search with backtracking that incrementally constructs candidate subsets and prunes branches that violate feasibility conditions or cannot improve the current best cost. Although the worst-case complexity remains exponential, the pruning strategy substantially reduces the number of candidate subsets explored in practice. The resulting procedure, summarized in Algorithm~\ref{algo:Backtracking_mfmc_selection}, returns the selected index set $\mathcal{S}$. Once the model subset has been selected, the continuous optimization problem is applied only to the selected models.

\normalem
\begin{algorithm}[!ht]
\label{algo:Backtracking_mfmc_selection}
\DontPrintSemicolon
\SetAlgoVlined
\SetKwProg{Fn}{Function}{}{}
\SetKwInOut{Input}{Input}
\SetKwInOut{Output}{Output}

\Input{%
Correlation coefficients $\boldsymbol{\rho}=(\rho_{1,1}, \ldots, \rho_{1,\widetilde K})$ and per-sample costs $\boldsymbol{C}=(C_1,\ldots,C_{\widetilde K})$.
}
\Output{%
  Selected model indices $\mathcal{S}$.
  
}
\hrule
 
\Fn{[idx\_for\_model, $\xi_{\min}$] = Model\_Selection\_Backtrack\,($\boldsymbol{\rho}, \boldsymbol{C}$)}{
Sort models by non-increasing $|\rho_{1,k}|$ with order $r$ and reorder $\boldsymbol{\rho}, \boldsymbol{C}$.\\

Initialize current\_idx = [1], $\xi_{\min}=1$, global\_idx = []\;  

\vspace{3mm}

[$\xi_{\text{min}}$, global\_idx] = \textbf{Backtrack} (current\_idx, $\xi_{\min}$, 2, global\_idx)\; 
    idx\_for\_model =  r(global\_idx)\;
}

\vspace{3mm}

\Fn{ [$\xi_{\text{min}}$, global\_idx] = \textbf{Backtrack} $\left(\text{current}\_\text{idx}, \, \xi, \,k_{\text{next}},\, global\_idx\right)$}{

  \If{$\xi \leq \xi_{\text{min}}\,$ }{
    $\xi_{\text{min}}=\xi$\;

    global\_idx = current\_idx\;
  }
  
  \If{$k_{\text{next}} > \widetilde{K}$}{ 
    \Return
  }
  
  \For{$k = k_{\text{next}}$ \KwTo $\widetilde{K}$}{ 
     prev\_idx = current\_idx(end)\;



    \If{%
      $\frac{\boldsymbol{C}({\text{prev}\_\text{idx}})}{\boldsymbol{C}(k)} > \frac{\boldsymbol{\rho}({\text{prev}\_\text{idx}})^2 - \boldsymbol{\rho}(k)^2}{\boldsymbol{\rho}(k)^2}$ 
    }{
        \textbf{continue}\;
    }

        
        Compute $\xi$ using \eqref{eq:MFMC_sampling_cost_efficiency} for  
      [current\textunderscore idx, $k$].

      \If{$\xi\ge \xi_{\text{min}}$ or $\,\xi>1$}{ 
    \textbf{continue}\;
    }
      
      [$\xi_{\text{min}}$, global\_idx] = \textbf{Backtrack} $(\, [\text{current}\_\text{idx},k], \, \xi,\, k+1,\, global\_idx)$.
  }
}

\vspace{3mm}

$\mathcal{S} = \{\text{idx}\_\text{for}\_\text{model}\}$ with size $K$.

\caption{Multi-fidelity Model Selection With Backtracking Pruning}
\end{algorithm}
\ULforem

\subsection{Recursive sample allocation}
A key structural property of problem \eqref{eq:Optimization_pb_sample_size} is that the weights $\boldsymbol{\alpha}$ appear only in the variance constraint $\mathcal{V}^{\mathrm{MF}}(\boldsymbol{\alpha},\boldsymbol{N})$, whereas the objective $\mathcal{W}^\text{MF} (\boldsymbol{N})$ depends only on the sample sizes. Therefore, for any fixed sample allocation $\boldsymbol{N}$, the optimization over $\boldsymbol{\alpha}$ can be performed independently. As shown in the proof of Theorem~\ref{thm:Sample_size_est}, this optimization admits the closed-form solution \eqref{eq:MFMC_RealValued_Sample_Size}. Substituting the resulting optimal weights into the variance constraint reduces the normalized variance to $\sum_{k=1}^K \Delta_k/N_k$ and yield the equivalent optimization problem
\begin{equation}\label{eq:Optimization_pb_sample_size_reduced_1}
    \begin{array}{ll}
    \displaystyle\min_{\boldsymbol{N} \in \mathbb{R}^{K}} 
        & \displaystyle \mathcal{W}^{\mathrm{MF}}(\boldsymbol{N}),
        \\[2mm]
    \text{subject to} 
        & \displaystyle \sum_{k=1}^K \frac{\Delta_k}{N_k} = \frac{\epsilon^2}{\sigma_1^2}, \\[2mm]
        & N_1 \ge 0,\quad N_{k-1} \le N_k,\;\; k=2,\ldots,K .
    \end{array}
\end{equation}
The reduced problem admits a natural sequential decomposition. Once the sample sizes on fidelity levels $1,\ldots,i-1$ are fixed, the remaining optimization depends on the previous decisions only through the residual normalized variance budget. This optimal-substructure property makes Bellman's principle of optimality \cite{Be:1957} applicable. Specifically, if $\boldsymbol N^*$ is an optimal solution of \eqref{eq:Optimization_pb_sample_size_reduced_1}, then, for any $i=1,\ldots,K$, the tail subsequence $(N_i^*,\ldots,N_K^*)$ is an optimal solution of the corresponding subproblem over fidelity levels $i,\ldots,K$. This leads to the following family of sequential subproblems
\begin{equation}\label{eq:Sequential_Optimization_revised}
    \begin{array}{ll}
    \displaystyle\min_{(N_i, \dots, N_K) \in \mathbb{R}^{K-i+1}} 
        & \displaystyle \sum_{k=i}^K C_k N_k, \\[4mm]
    \text{subject to} 
        & \displaystyle \sum_{k=i}^K \frac{\Delta_k}{N_k} = R_i, \\[2mm]
        & N_i \ge N_{i-1}, \quad N_{k-1} \le N_k,\;\; k=i+1,\ldots,K.
    \end{array}
\end{equation}
where $R_i$ denotes the residual normalized variance to be allocated among fidelity levels $i, \ldots, K$. With the initialization $R_1=\epsilon^2/\sigma_1^2$, it satisfies the recursion
\begin{equation}\label{eq:R_k_def}
    R_i = R_{i-1} - \frac{\Delta_{i-1}}{N_{i-1}}, \quad i=2, \dots, K.
\end{equation}

Compared with the original optimization problem \eqref{eq:Optimization_pb_sample_size_reduced_1}, the sequential subproblem \eqref{eq:Sequential_Optimization_revised} differs only in two respects: the prescribed variance tolerance is replaced by the residual normalized variance budget $R_i$, and the lower-bound constraint $N_1\ge 0$ is replaced by the coupling constraint $N_i\ge N_{i-1}$ to preserve the global monotonicity of the sample allocation. The latter coupling prevents the direct application of Theorem~\ref{thm:Sample_size_est}. We therefore consider the following relaxed subproblem, obtained by replacing the constraint $N_i\ge N_{i-1}$ with $N_i\ge0$ while retaining the remaining ordering constraints
\begin{equation}\label{eq:Sequential_Optimization_relaxed}
    \begin{array}{ll}
    \displaystyle\min_{(N_i, \dots, N_K) \in \mathbb{R}^{K-i+1}} 
        & \displaystyle \sum_{k=i}^K C_k N_k, \\[4mm]
    \text{subject to} 
        & \displaystyle \sum_{k=i}^K \frac{\Delta_k}{N_k} = R_i, \\[2mm]
        & N_i \ge 0, \quad N_{k-1} \le N_k, \;\; k=i+1, \dots, K.
    \end{array}
\end{equation}

The relaxed subproblem has the same mathematical structure as the reduced optimization problem \eqref{eq:Optimization_pb_sample_size_reduced_1}. Therefore, Theorem~\ref{thm:Sample_size_est} applies directly, yielding
\begin{equation}\label{eq:MFMC_Recursive_RealValued_Sample_Size}
    N_k^{*,(i)} = \sqrt{\frac{\Delta_k}{C_k}} \frac{S_i}{R_i}, \quad k=i,\ldots,K,
\end{equation}
where
\begin{equation}
\label{eq:S_i}
S_i := \sum_{j=i}^K \sqrt{C_j\Delta_j}.
\end{equation}

Theorem~\ref{thm:MFMC_Iterative_RealValued_Sample_Size} shows that, although the relaxed subproblems start from different fidelity levels, they assign the same optimal sample size to every common level. Consequently, the stage-wise solutions can be assembled into a single sample allocation, which coincides with the global optimum of the original MFMC allocation problem. The proof is deferred to the Appendix.

\begin{theorem}
\label{thm:MFMC_Iterative_RealValued_Sample_Size}
Under the assumptions of Theorem~\ref{thm:Sample_size_est}, let $(N_i^{*,(i)},\ldots, N_K^{*,(i)})$ denote the optimal solution of the relaxed
subproblem \eqref{eq:Sequential_Optimization_relaxed}, whose explicit form is
given by \eqref{eq:MFMC_Recursive_RealValued_Sample_Size}. Then
\begin{enumerate}
    \item \textbf{Invariance.}
    The scaling ratio is invariant across subproblems
    \begin{equation}
    \label{eq:R_k_S_k_1}
        \frac{S_i}{R_i} = \frac{S_{i-1}}{R_{i-1}}, 
        \qquad i=2,\ldots,K.
    \end{equation}

    \item \textbf{Cross-subproblem consistency.}
    For every level $k$ and every pair of admissible starting indices $i,j\le k$, the corresponding sample sizes satisfy
\[
N_k^{*,(i)}=N_k^{*,(j)}.
\]
Hence the stage-aligned sequence $(N_1^{*,(1)},\ldots, N_K^{*,(K)})$ is well defined and independent of the choice of subproblem index.
    \item \textbf{Global optimality.}
    The stage-wise sequence coincides with the global MFMC optimum
    \[
        N_k^{*,(k)} = N_k^*, \quad k=1,\ldots,K.
    \]

\end{enumerate}
\end{theorem}

The remaining step is to justify that the relaxation introduced in
\eqref{eq:Sequential_Optimization_relaxed} yields the same solution as the sequential subproblem
\eqref{eq:Sequential_Optimization_revised}. Since $N_k>0$ for all $k$, the change of variables
$y_k=N_k^{-1}$ transforms the sequential subproblem into a convex optimization problem with a strictly convex objective, a convex feasible set, and affine constraints. Consequently, the sequential subproblem admits a unique minimizer. Moreover, under the assumptions of Theorem~\ref{thm:Sample_size_est}, the minimizer of the relaxed subproblem is strictly increasing and therefore satisfies the monotonicity constraint in
\eqref{eq:Sequential_Optimization_revised}. Hence, it is feasible for the sequential subproblem. Since the feasible set of the sequential subproblem is contained in that of the relaxed problem, the relaxed minimizer is also optimal for the sequential subproblem. By uniqueness, the relaxed and sequential subproblems have the same minimizer.

Consequently, the continuous MFMC optimum can be constructed sequentially. Starting from
$R_1=\epsilon^2/\sigma_1^2$, the sample size $N_k^{*,(k)}$ is computed from
\eqref{eq:MFMC_Recursive_RealValued_Sample_Size} at each stage, after which the residual variance budget is updated according to $R_{k+1} = R_k-\Delta_k/N_k^{*,(k)}$. This recursive construction provides a natural framework  for the integer-valued sample allocation developed in the next section.

\section{Recursive integer-valued sample allocation}\label{sec:Iterative_IntegerValued_Sample_Size}
Since sample sizes must be integers in practice, the continuous MFMC allocation is generally infeasible. A natural approach is to round the optimal continuous solution. In \cite{PeWiGu:2016}, which considers variance minimization under a fixed cost budget, floor-based rounding satisfies the prescribed budget constraint but may assign zero samples to the high-fidelity model when the budget is limited, resulting in an infeasible estimator. To address this issue, \cite{GrGuJuWa:2023} proposed a hybrid floor--ceiling rounding strategy that selectively rounds sample sizes up or down to ensure feasibility. Although this approach avoids zero-sample allocations, it remains a post-processing procedure applied to the continuous solution. For the present variance-constrained cost minimization formulation, we use the ceiling rule as a simple baseline to obtain an integer-valued sample allocation. It guarantees $\lceil N_k^*\rceil\ge1$ for every fidelity level and therefore satisfies the prescribed variance constraint $\mathcal{V}_n^{\mathrm{MF}}\le\epsilon^2/\sigma_1^2$. This construction implicitly assumes $\epsilon\le\sigma_1$, analogous to the minimum cost budget requirement $p\ge\sum_kC_k$ in \cite{GrGuJuWa:2023}. The resulting variance and cost satisfy
%
\begin{equation}\label{eq:bounds_for_ceil}
\begin{aligned}
    \mathcal{V}^{\mathrm{MF}}_n \in \left(\sum_{k=1}^K\frac{\Delta_{k}}{N_k^*+1},\frac{\epsilon^2}{\sigma_1^2}\right], \qquad
\mathcal{W}^{\text{MF}}\in \left[\frac{\sigma_1^2}{\epsilon^2}\left(\sum_{k=1}^K\sqrt{C_k\Delta_k}\right)^2, \frac{\sigma_1^2}{\epsilon^2}\left(\sum_{k=1}^K\sqrt{C_k\Delta_k}\right)^2+\sum_{k=1}^K C_k\right).
\end{aligned}
\end{equation}
Although the ceiling strategy guaranties feasibility, it over-allocates samples relative to the continuous optimum. As a result, the computational cost can exceed the continuous optimum by as much as $\sum_{k=1}^K C_k$ and the available variance budget $\epsilon^2/\sigma_1^2$ is not fully used. This inefficiency is particularly pronounced in the pre-asymptotic regime, where the optimal sample sizes are small and rounding errors become relatively significant.  

To better use the available variance budget, we develop a recursive integer-valued sample allocation strategy based on the sequential construction introduced in the previous section. Instead of rounding the continuous allocation only after the recursion is completed, the rounding operation is incorporated into the recursion itself. At each stage, the current sample allocation is rounded upward, and the corresponding variance contribution is deducted from the remaining variance budget before computing the allocation for the next fidelity level. In this way, each subsequent allocation accounts for the rounding decisions made at the preceding levels. This defines the following recursive continuous sample-size proxies
\begin{equation}\label{eq:MFMC_New_IntegerValued_Sample_Size}
    M_1^* = \sqrt{\frac{\Delta_1}{C_1}}\frac{\sum_{j=1}^K\sqrt{C_j\Delta_j}}{\frac{\epsilon^2}{\sigma_1^2}}, 
    \qquad 
    M_k^* = \sqrt{\frac{\Delta_k}{C_k}}\frac{\sum_{j=k}^K\sqrt{C_j\Delta_j}}{\frac{\epsilon^2}{\sigma_1^2}-\sum_{j=1}^{k-1}\frac{\Delta_j}{\left\lceil M_j^* \right\rceil}}, 
    \quad k = 2,\ldots, K.
\end{equation}

Theorem~\ref{thm:MFMC_New_IntegerValued_Variance_Cost}  compares the proposed recursive allocation with the continuous optimum and the direct ceiling-rounding strategy. It shows that the recursive construction generates no larger continuous sample-size proxies than the continuous optimum and, after ceiling, produces an integer allocation whose variance remains feasible while whose computational cost is never larger than that obtained by direct ceiling rounding.
\begin{theorem}
\label{thm:MFMC_New_IntegerValued_Variance_Cost}
Under the assumptions of Theorem~\ref{thm:Sample_size_est}, let $M_k^*$
be the continuous sample-size proxies generated by the recursive scheme
\eqref{eq:MFMC_New_IntegerValued_Sample_Size}, and let
$N_k^*$
be the continuous MFMC optimal allocation given by
\eqref{eq:MFMC_RealValued_Sample_Size}.
Then
\[
M_k^*\le N_k^*,\qquad k=1,\ldots,K.
\]

Consequently, the corresponding integer-valued allocations satisfy
\begin{align}
\label{eq:Iterative_integer_sample_size_variance_bound}
&\sum_{k=1}^K \frac{\Delta_k}{\left\lceil N_k^* \right\rceil}
\;\le\;
\sum_{k=1}^K \frac{\Delta_k}{\left\lceil M_k^* \right\rceil}
\;\le\;
\frac{\epsilon^2}{\sigma_1^2},\\
\label{eq:Iterative_integer_sample_size_cost_bound}
&\frac{\sigma_1^2}{\epsilon^2}\left(\sum_{k=1}^K\sqrt{C_k\Delta_k}\right)^2
    \;\le\;
    \sum_{k=1}^K C_k \left\lceil M_k^* \right\rceil
    \;\le\;
    \sum_{k=1}^K C_k \left\lceil N_k^* \right\rceil.
\end{align}
\end{theorem}
The integer-valued allocation produced by the proposed method is
$(\lceil M_1^*\rceil,\ldots,\lceil M_K^*\rceil)$,
which is then used in the standard MFMC estimator described in
Algorithm~\ref{algo:MFMC_Algo}.

\begin{algorithm}[!ht]
\DontPrintSemicolon
\label{algo:MFMC_Algo}

\KwIn{Correlations $\rho_{1,k}$, costs $C_k$, weights $\alpha_k$, and prescribed tolerance $\epsilon$.}

\KwOut{Integer sample allocation $\lceil\boldsymbol{M}^*\rceil$ and MFMC estimator $A^{\mathrm{MF}}$.}

\vspace{1ex}
\hrule
\vspace{1ex}

Determine the integer sample allocation
$\lceil\boldsymbol{M}^*\rceil
=(\lceil M_1^*\rceil,\ldots,\lceil M_K^*\rceil)$
using the recursive scheme
\eqref{eq:MFMC_New_IntegerValued_Sample_Size}.\;

Evaluate the high-fidelity model $u_1$ using
$\lceil M_1^*\rceil$
i.i.d.\ samples and compute
$A^{\mathrm{MC}}_{1,\lceil M_1^*\rceil}$.\;

\For{$k=2,\ldots,K$}{

Evaluate the low-fidelity model $u_k$ on the shared
$\lceil M_{k-1}^*\rceil$ samples and compute
$A^{\mathrm{MC}}_{k,\lceil M_{k-1}^*\rceil}$.\;

Evaluate $u_k$ on the remaining
$\lceil M_k^*\rceil-\lceil M_{k-1}^*\rceil$
independent samples and compute the corresponding sample average
$A^{\mathrm{MC}}_{k,\lceil M_k^*\rceil\setminus\lceil M_{k-1}^*\rceil}$.\;
}

Construct the MFMC estimator
$A^{\mathrm{MF}}$
according to
\eqref{eq:MFMC_estimator_independent}.

\caption{Recursive multifidelity Monte Carlo sampling}
\end{algorithm}

The proposed recursive construction is not restricted to the variance-constrained cost minimization formulation considered here. It applies equally to the equivalent formulation of variance minimization under a fixed cost budget. Furthermore, as discussed in the previous section, the reduced continuous optimization problem \eqref{eq:Optimization_pb_sample_size_reduced_1} has essentially the same mathematical structure as the sample allocation problem in multilevel Monte Carlo \cite{Gi:2008,Gi:2015}. Consequently, the same forward recursive rounding strategy extends naturally to MLMC, yielding integer-valued sample allocations with guaranteed variance control.

\section{Numerical results}\label{sec:Num_Result}
We evaluate the proposed recursive integer-valued MFMC allocation strategy on four representative test problems, including a plasma equilibrium model \cite{ElLiSa:2023,ElLiSa:2026}, a tubular reactor problem \cite{PeWiGu:2016}, the Ishigami benchmark \cite{QiPeOMVe:2018}, and a heterogeneous elasticity model \cite{GoGeElJa:2020}. These examples span a broad range of application domains, model correlations, and computational costs, and are therefore well suited for assessing the robustness of integer-valued sample allocation schemes. For each problem, we compare the proposed recursive ceiling-based allocation with three integer-valued approaches: direct flooring \cite{PeWiGu:2016}, modified flooring \cite{GrGuJuWa:2023}, and direct ceiling rounding. Results are presented for two various prescribed relative variance tolerances with $ \epsilon\le \sigma_1$. The corresponding continuous MFMC optimal allocation is also reported as a reference, attaining the prescribed variance with minimal expected computational cost discussed in \eqref{eq:MFMC_sampling_cost}. All computations were performed in {\tt MATLAB} R2024a on a MacBook Pro with a 14-core Apple M4 Pro processor and 24~GB of RAM.

Table~\ref{tab:MFMC_integer_comparison_Eg1} summarizes the results for the plasma equilibrium problem with five fidelity levels. We first consider the tolerance $\epsilon^2/\sigma_1^2=5.9276\times10^{-4}$, for which the continuous optimal allocation satisfies $N_1^*<1$. The continuous MFMC allocation attains the prescribed variance at the minimum cost of $73.03$.  Direct flooring assigns zero samples to the highest-fidelity model, rendering the resulting MFMC estimator infeasible. The modified flooring strategy restores feasibility by enforcing at least one sample for each selected model. However, despite consuming nearly the entire available computational budget, it substantially inflates the resulting variance. In this example, the first three models receive identical sample sizes. Since MFMC variance reduction arises from the differences between consecutive sample sets in \eqref{eq:MFMC_estimator}, assigning identical sample sizes to adjacent fidelity levels provides no additional variance reduction while still incurring computational cost\footnote{The occurrence of identical consecutive sample sizes is not specific to any particular rounding strategy, but instead results from performing integer sample allocation on a fixed fidelity hierarchy. In practice, the choice of fidelity levels may itself depend on the prescribed variance tolerance or computational budget. Joint optimization of fidelity selection and integer sample allocation therefore leads to a more general discrete optimization problem, which is beyond the scope of the present work.}. Consequently, the resulting estimator attains a variance of $2.4834\times10^{-2}$, more than forty times larger than the prescribed tolerance. Ceiling-based rounding avoids the feasibility issue of flooring and satisfies the prescribed variance constraint. However, independently rounding each continuous sample size upward introduces unnecessary oversampling, increasing the computational cost to $81.71$, corresponding to an overhead of $11.9\%$ relative to the continuous optimum. The proposed recursive ceiling strategy addresses this limitation by sequentially updating the remaining variance budget after each rounding step. As a result, it preserves the prescribed variance almost exactly while reducing the total cost to $77.64$, recovering approximately one half of the excess cost introduced by direct ceiling rounding.

For the more stringent tolerance, $\epsilon^2/\sigma_1^2=5\times10^{-5}$, all methods remain feasible because the continuous optimal allocation assigns more than one sample to the highest-fidelity model. Direct flooring underutilizes the computational budget by approximately 3.8\%, resulting in a variance about 3.9\% above the prescribed tolerance. In contrast, direct ceiling satisfies the variance constraint conservatively, leaving approximately 4.4\% of the available variance budget unused while incurring a 4.7\% increase in computational cost. The proposed recursive allocation again nearly attains the prescribed variance, with a total cost within 1.5\% of the continuous optimum. Together with the previous tolerance case, these results demonstrate that the proposed recursive ceiling strategy consistently achieves near-optimal computational cost while accurately satisfying the prescribed variance constraint across different tolerance levels.

\begin{table}[ht]
\centering
\scalebox{0.9}{
\begin{tabular}{cc} 

\begin{minipage}{0.2\linewidth}
\centering
\small
\begin{tabular}{|c|c|}
\hline
\textbf{Model} & \textbf{$\rho_{1,k}$} \\
\hline
1 & 1 \\
2 & 9.9977e-01 \\
3 & 9.9925e-01 \\
4 & 9.9728e-01 \\
5 & 9.8390e-01 \\
\hline
& $C_k$ \\ \hline
1 & 73 \\
2 & 7.0318e-03 \\
3 & 1.4018e-03 \\
4 & 5.0613e-04 \\
5 & 2.6803e-04 \\
\hline
\end{tabular}
\end{minipage}
&
\begin{minipage}{0.8\linewidth}
\centering
\small
\begin{tabular}{|c|p{5.5cm}|c|c|}
\hline
\textbf{Method} & \textbf{Sample size} & \textbf{Cost} & \textbf{$\mathcal{V}^{\mathrm{MF}}_n$} \\ \hline

\multicolumn{4}{|c|}{\textbf{Tolerance: $\epsilon^2/\sigma_1^2 = 5.9276\times 10^{-4}$}} \\ \hline
Real-valued 
& [8.81e-01, 1.35e+02, 5.88e+02, 2.54e+03, 2.11e+04]
& 7.3030e+01 & 5.9276e-04  \\ \hline
Direct floor 
& [0, 134, 587, 2540, 21094]
& 8.7045e+00 & $\infty$ \\ \hline
Modified 
& [1, 1, 1, 7, 62]
&7.3030e+01 & 2.4834e-02  \\ \hline
Direct ceil
& [1, 135, 588, 2541, 21095]
& 8.1714e+01 &5.9275e-04\\ \hline
Iterative ceil
& [1, 72, 313, 1353, 11229]
& 7.7640e+01&5.9276e-04\\ \hline



\multicolumn{4}{|c|}{\textbf{Tolerance: $\epsilon^2/\sigma_1^2 = 5\times 10^{-5} $}} \\ \hline
Real-valued 
& [1.04e+01, 1.60e+03, 6.97e+03, 3.01e+04, 2.50e+05]
& 8.6578e+02  &5.0000e-05 \\ \hline
Direct floor 
& [10, 1599, 6970, 30114, 250079]
&8.3328e+02  &5.1960e-05   \\ \hline
Direct ceil
& [11, 1600, 6971, 30115, 250080]
& 9.0629e+02 & 4.7779e-05\\ \hline
Iterative ceil
& [11, 1166, 5079, 21943, 182217]
& 8.7826e+02 & 5.0000e-05 \\ \hline

\end{tabular}
\end{minipage}

\end{tabular}
}
\caption{Model parameters and MFMC sample allocations for the plasma equilibrium problem \cite{ElLiSa:2023,ElLiSa:2026}. Continuous and integer-valued MFMC allocations are compared for two prescribed relative variance tolerances, $\epsilon^2/\sigma_1^2=5.9276\times10^{-4}$ and $5\times10^{-5}$. All five candidate models are included in the MFMC estimator.}
\label{tab:MFMC_integer_comparison_Eg1}
\end{table}

\begin{figure}[!t]\centering
\begin{tabular}{cc}
\includegraphics[width=0.48\linewidth]{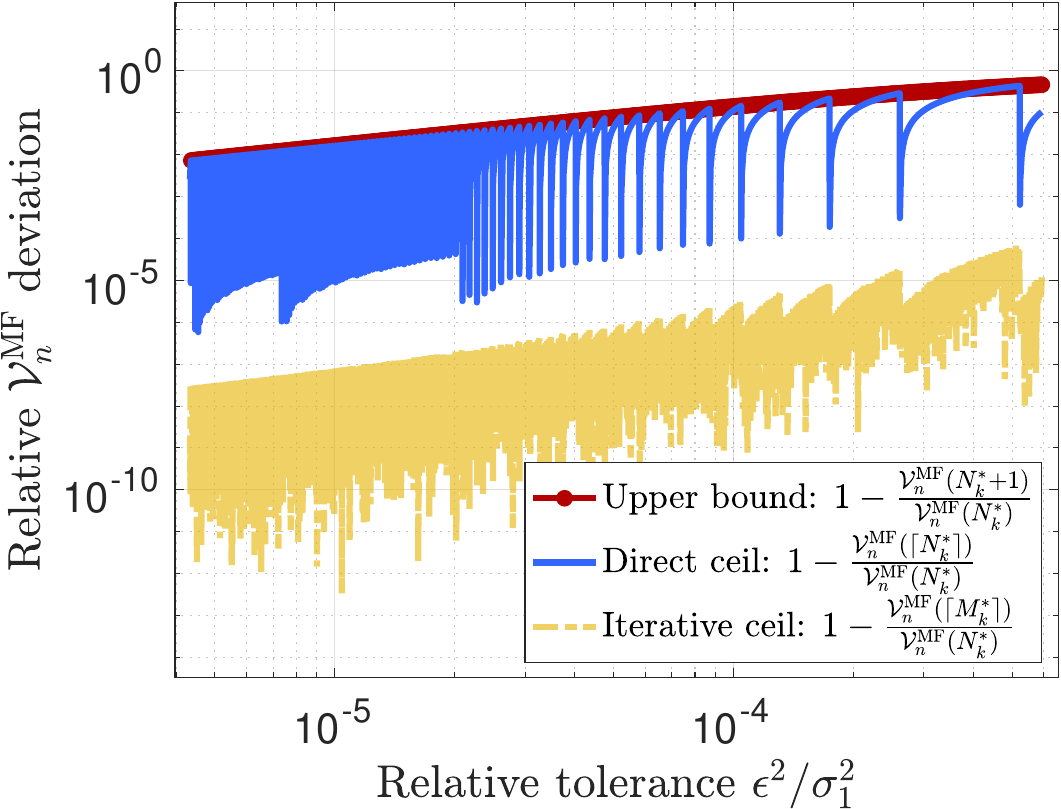} &
\includegraphics[width=0.48\linewidth]{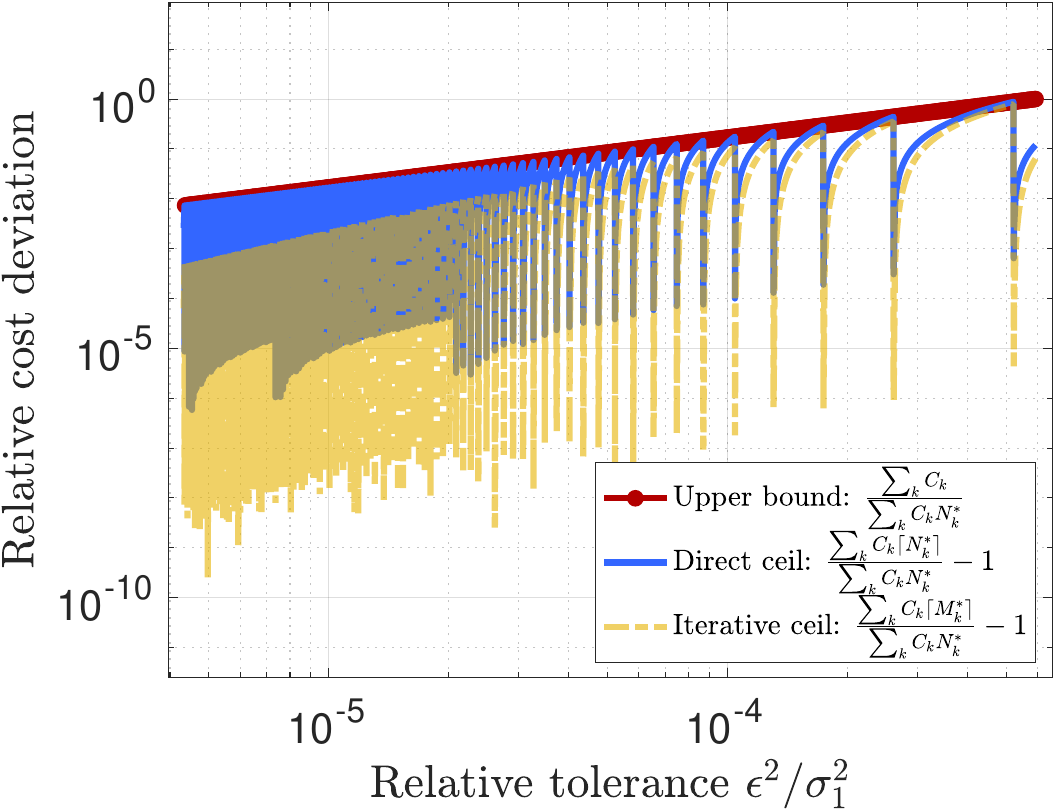}\\
\includegraphics[width=0.48\linewidth]{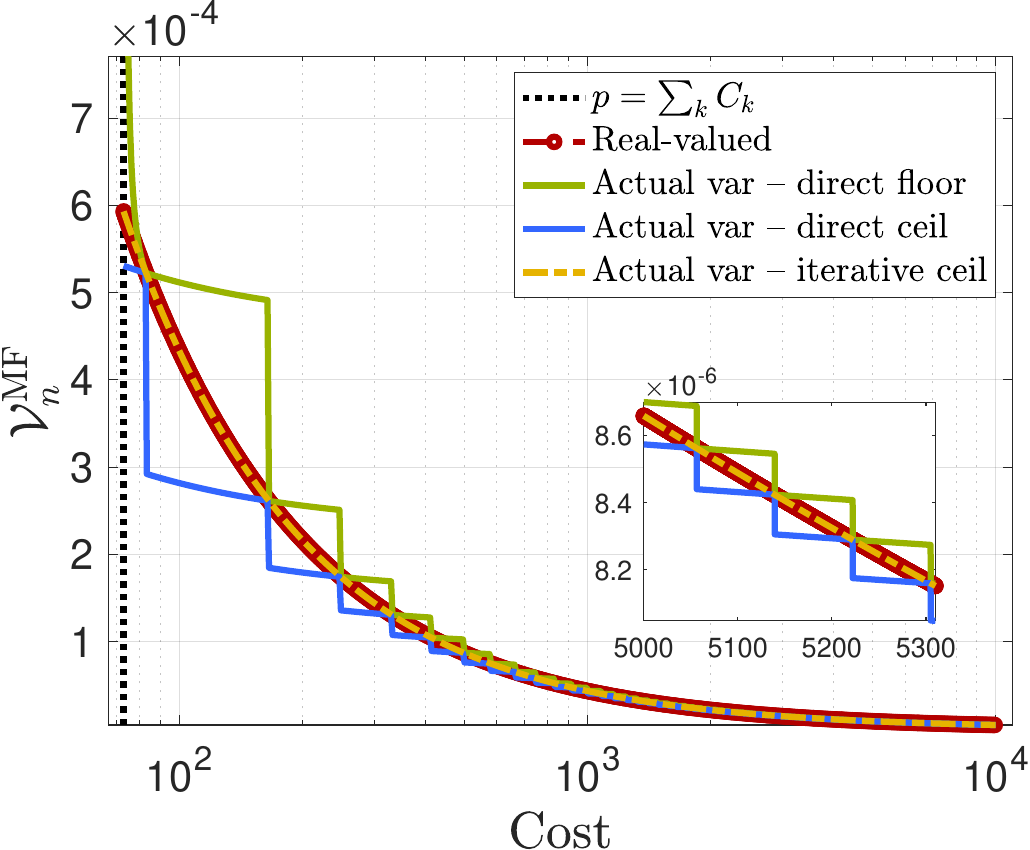} &
\includegraphics[width=0.48\linewidth]{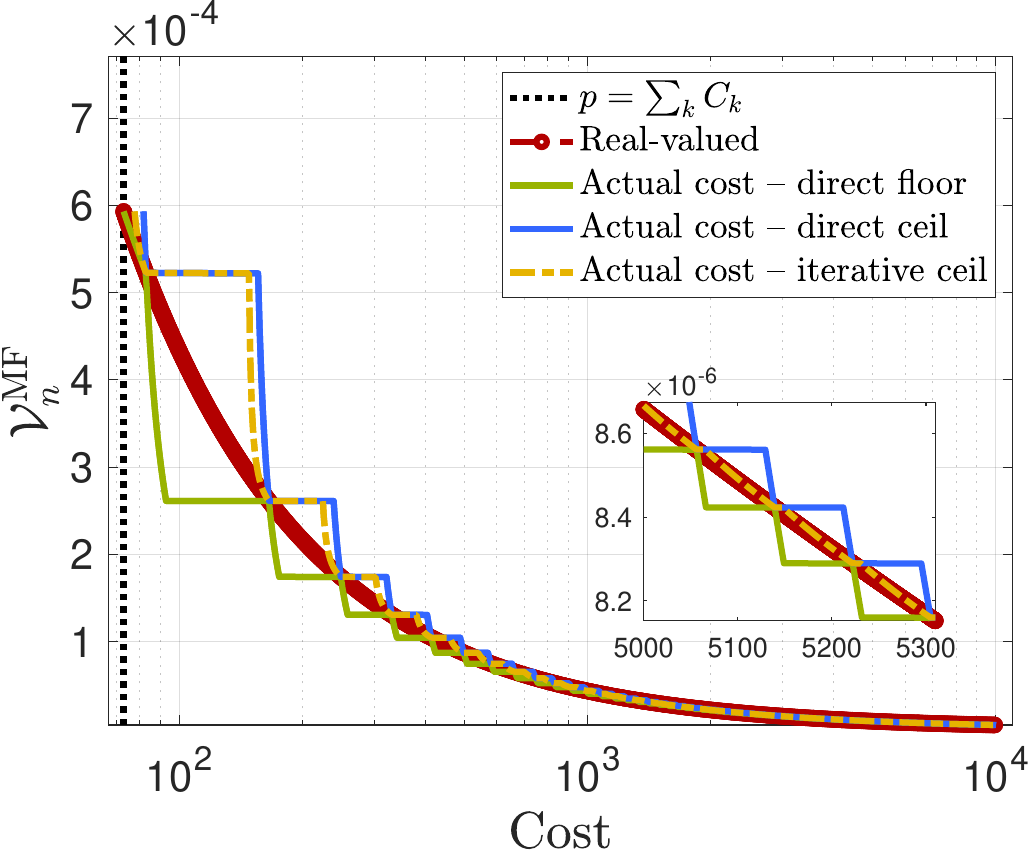}\\
\end{tabular}

\caption{Comparison of integer MFMC allocation strategies over a range of prescribed relative variance tolerances $\mathcal{V}_n^{\mathrm{MF}}\in[4.33\times10^{-6},\,5.93\times10^{-4}]$. Top left: Relative variance deviation from the prescribed tolerance. Top right: Relative cost increase with respect to the continuous MFMC optimum. Bottom left: Achieved variance as a function of computational cost. Bottom right: Computational cost as a function of achieved variance.}
\label{fig:Eg1} 
\end{figure}

Figure~\ref{fig:Eg1} further illustrates the performance of the different integer-valued allocation strategies over a wide range of prescribed relative variance tolerances, $\mathcal{V}_n^{\mathrm{MF}}\in[4.33\times10^{-6},\,5.93\times10^{-4}]$, corresponding to continuous computational budgets ranging from $\sum_k C_k$ to $10^4$. The top panels show the relative deviations in variance and cost from the continuous MFMC optimum, while the bottom panels compare the achieved variance at fixed cost and the achieved cost at fixed variance.

The recursive allocation consistently tracks the continuous optimum more closely than direct ceiling rounding. As shown in the top-left panel, recursive scheme achieves variance values that remain very close to the prescribed tolerance throughout the entire range of test cases, whereas direct ceiling over-allocates samples and therefore leaves part of the available variance budget unused. This improved utilization of the variance budget translates directly into lower computational cost, as illustrated in the top-right panel, where the recursive allocation remains closer to the continuous optimum than direct ceiling, with both strategies satisfying the upper bound established in \eqref{eq:bounds_for_ceil}.

The bottom panels provide a complementary view of the integer allocation behavior by considering the achieved variance for a prescribed computational budget and the achieved cost for a prescribed variance tolerance. The observed sawtooth patterns are a direct consequence of integer-valued sample allocations. As the prescribed tolerance varies continuously, the corresponding continuous optimal sample sizes change smoothly, whereas the rounded integer allocations remain unchanged over intervals until one or more continuous sample sizes cross integer thresholds. At such transition points, additional samples are introduced, resulting in discrete changes in both computational cost and achieved variance. From a Pareto-optimality perspective, the continuous MFMC allocation lies on the optimal cost--variance frontier, while integer rounding introduces deviations from this frontier. Direct flooring reduces computational cost at the expense of violating the prescribed variance constraint, whereas direct ceiling satisfies the variance constraint conservatively but introduces additional computational overhead.  By incorporating rounding decisions into the recursive residual-variance update, the proposed method remains much closer to the Pareto frontier. The advantage of the recursive strategy is most evident in the pre-asymptotic regime, where the prescribed tolerance is relatively large and the optimal sample sizes are small, making the allocation more sensitive to integer rounding effects. As the prescribed tolerance decreases and the continuous optimal sample sizes grow, the relative impact of integer rounding diminishes, and the performance differences among the integer-valued allocation strategies become smaller.

\begin{table}[ht]
\centering
\scalebox{0.9}{
\begin{tabular}{cc} 

\begin{minipage}{0.2\linewidth}
\centering
\small
\begin{tabular}{|c|c|}
\hline
\textbf{Model} & \textbf{$\rho_{1,k}$} \\
\hline
1 & 1 \\
2 & 9.9999e-01 \\
3 & 9.9997e-01 \\
4 & 9.9583e-01 \\
\hline
&$C_k$\\
\hline
1 & 4.4395e+01 \\
2 & 6.8409e-01 \\
3 & 2.9937e-01 \\
4 & 1.9908e-04 \\
\hline
\end{tabular}
\end{minipage}
&
\begin{minipage}{0.8\linewidth}
\centering
\small
\begin{tabular}{|c|p{5.5cm}|c|c|c|}
\hline
\textbf{Method} & \textbf{Sample size} & \textbf{Cost} & \textbf{$\mathcal{V}^{\mathrm{MF}}_n$}\\ \hline

\multicolumn{4}{|c|}{\textbf{Tolerance: $\epsilon^2/\sigma_1^2 = 2.1987\times 10^{-4}$}} \\ \hline
Real-valued 
& [3.33e-01, 2.92e+00, 7.61e+01, 3.23e+04]
& 4.6000e+01 & 2.1987e-04 \\ \hline
Direct floor 
& [0, 2, 76, 32282]
& 3.0547e+01 & $\infty$ \\ \hline
Modified 
& [1, 1, 2, 1018]
& 4.5880e+01 & 5.1658e-03 \\ \hline
Direct ceil
& [1, 3, 77, 32283]
& 7.5926e+01 &1.7112e-04 \\ \hline
Iterative ceil
& [1, 3, 57, 23763]
& 6.8242e+01 & 2.1987e-04\\ \hline

\multicolumn{4}{|c|}{\textbf{Tolerance: $\epsilon^2/\sigma_1^2 = 5 \times 10^{-5}$}} \\ \hline
Real-valued 
& [1.47e+00, 1.28e+01, 3.35e+02, 1.42e+05]
& 2.0228e+02 &5.0000e-05 \\ \hline
Direct floor 
& [1, 12, 334, 141959]
& 1.8085e+02 & 5.7694e-05 \\ \hline
Direct ceil 
& [2, 13, 335, 141960]
& 2.2623e+02 & 4.5642e-05\\ \hline
Iterative ceil
& [2, 12, 296, 125393]
& 2.1058e+02 & 5.0000e-05\\ \hline
\end{tabular}
\end{minipage}

\end{tabular}
}
\caption{Model parameters and MFMC sample allocations for the tubular reactor
example \cite{PeWiGu:2016}. Continuous and integer-valued MFMC allocations are compared for two prescribed relative variance tolerances, $\epsilon^2/\sigma_1^2 = 2.1987\times 10^{-4}$ and $5\times 10^{-5}$. All four candidate models are selected.}
\label{tab:MFMC_integer_comparison_Eg2}
\end{table}

\begin{figure}[!t]\centering
\begin{tabular}{cc}
\includegraphics[width=0.48\linewidth]{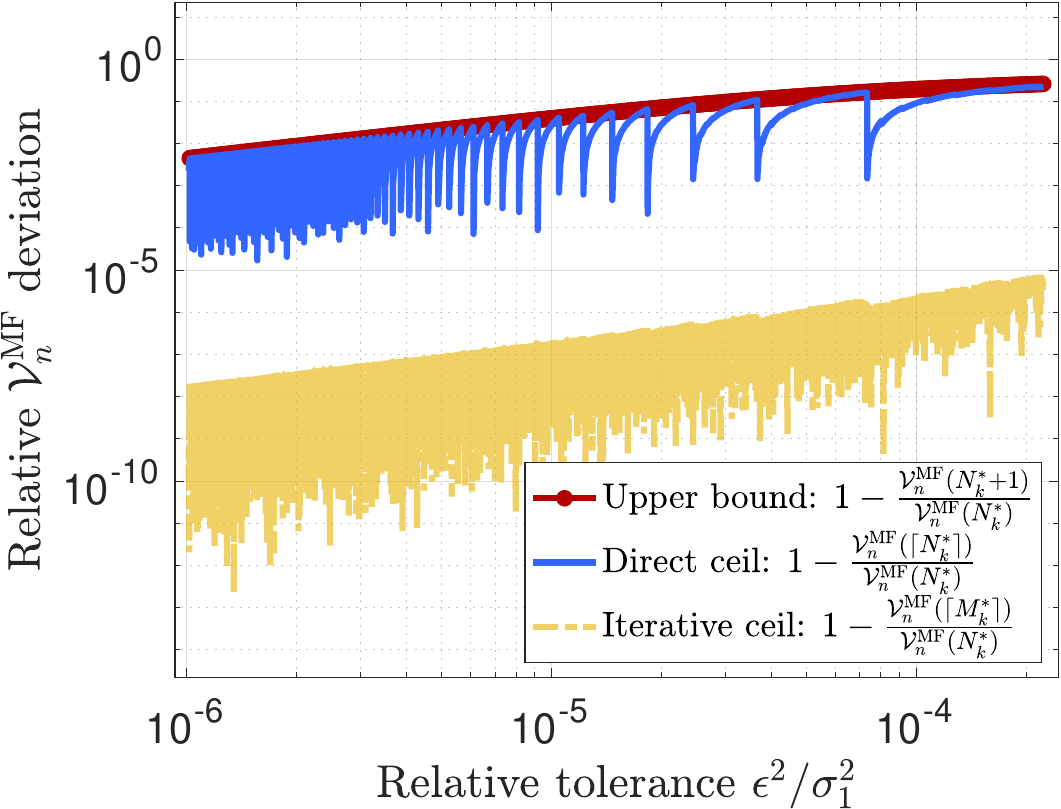}&
\includegraphics[width=0.48\linewidth]{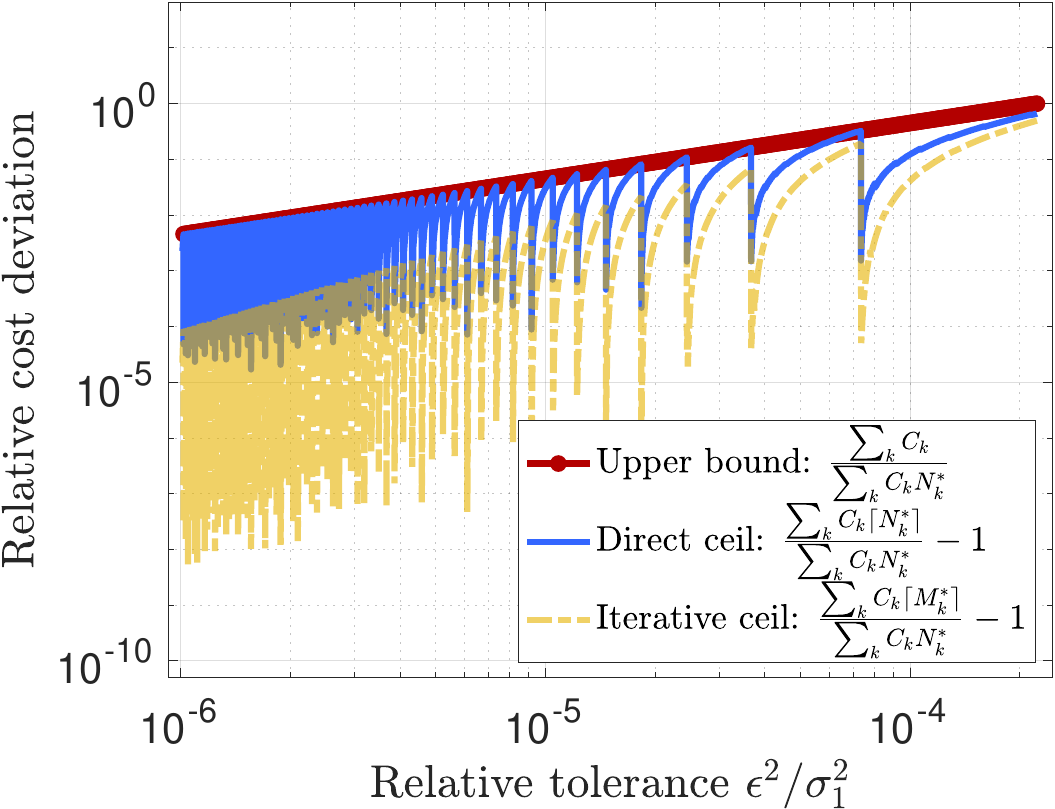} \\
\includegraphics[width=0.48\linewidth]{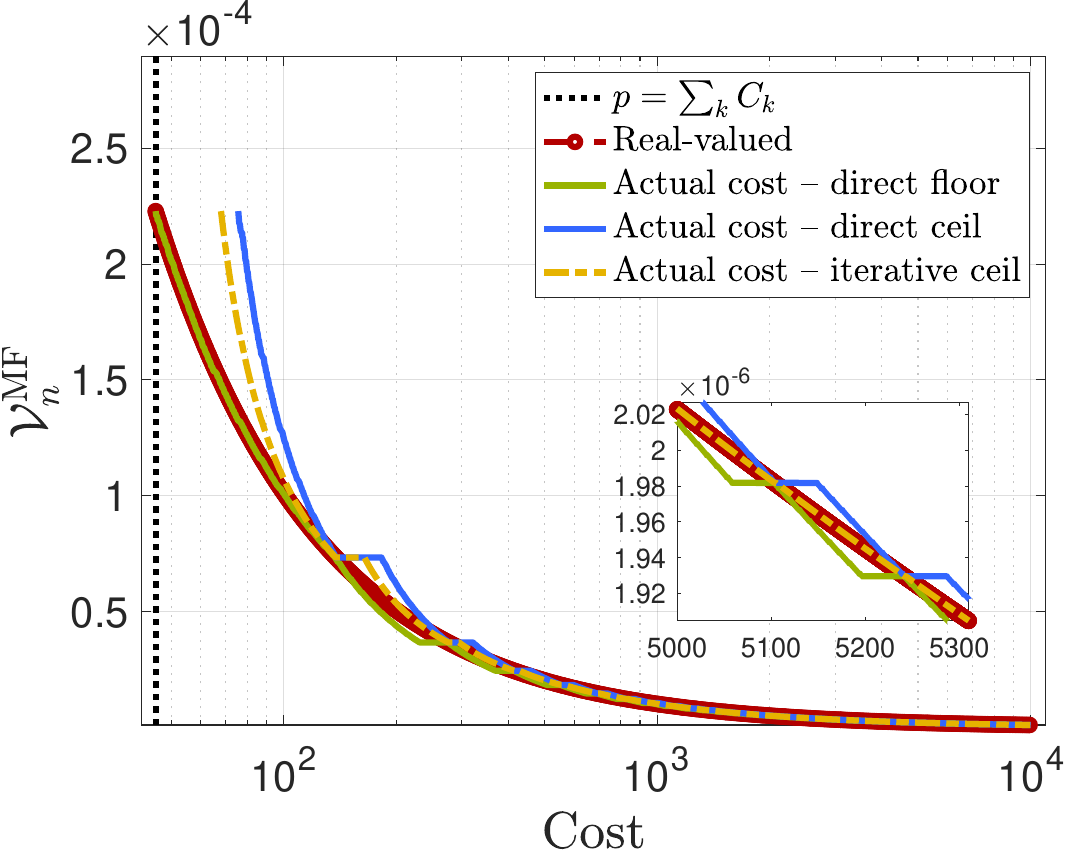} &
\includegraphics[width=0.48\linewidth]{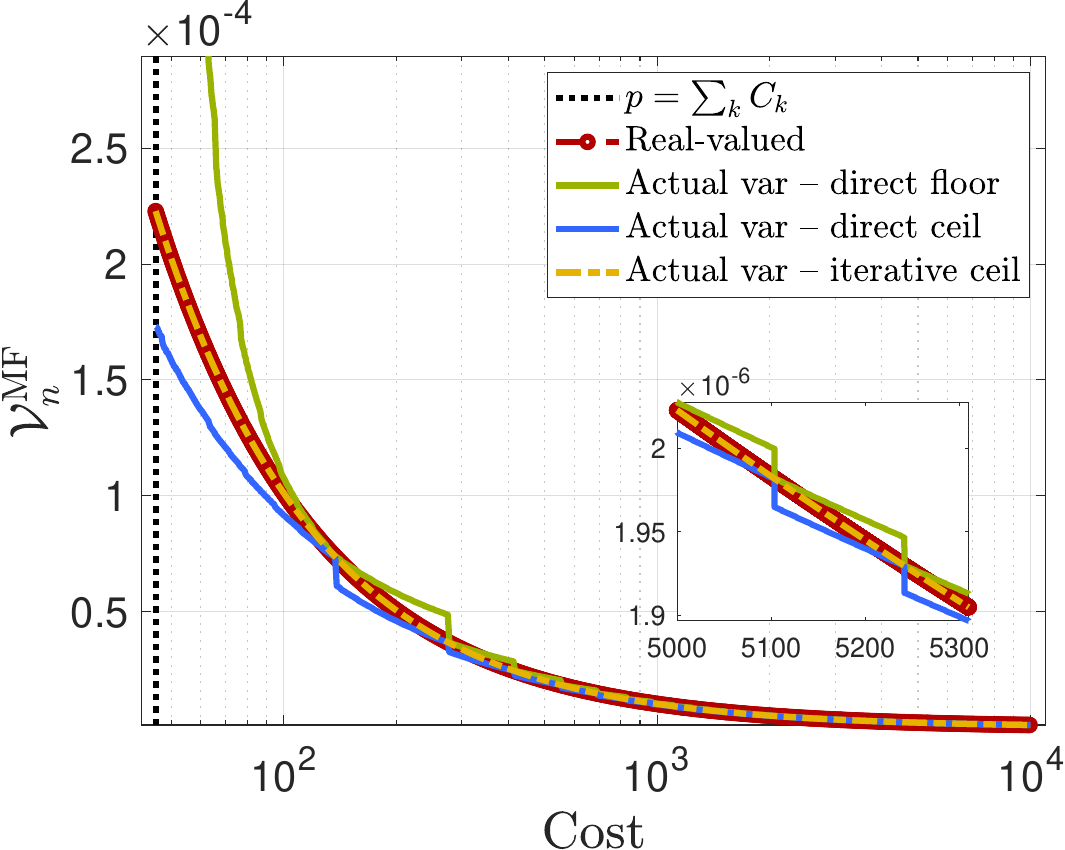}\\
\end{tabular}
\caption{Comparison of integer MFMC allocation strategies over a range of prescribed relative variance tolerances $\mathcal{V}_n^{\mathrm{MF}}\in[1.01\times 10^{-6},\,2.23\times 10^{-4}]$. Top left: Relative variance deviation from the prescribed tolerance. Top right: Relative cost increase with respect to the continuous MFMC optimum. Bottom left: Achieved variance as a function of computational cost. Bottom right: Computational cost as a function of achieved variance.}
\label{fig:Eg2} 
\end{figure}

\begin{table}[ht]
\centering
\scalebox{0.92}{
\begin{tabular}{cc} 

\begin{minipage}{0.25\linewidth}
\centering
\small
\begin{tabular}{|c|c|}
\hline
\textbf{Model} & \textbf{$\rho_{1,k}$} \\
\hline
1 & 1 \\
2 & 9.9970e-01 \\
3 & 9.4650e-01 \\
\hline
&$C_k$\\
\hline
1 & 1 \\
2 & 0.05 \\
3 & 0.001 \\
\hline
\end{tabular}
\end{minipage}
&
\begin{minipage}{0.73\linewidth}
\centering
\small
\begin{tabular}{|c|p{4.9cm}|c|c|}
\hline
\textbf{Method} & \textbf{Sample size} & \textbf{Cost} & \textbf{$\mathcal{V}^{\mathrm{MF}}_n$} \\ \hline

\multicolumn{4}{|c|}{\textbf{Tolerance: $\epsilon^2/\sigma_1^2 = 1.4519\times 10^{-2} $}} \\ \hline

Real-valued 
& [2.13e-01, 1.25e+01, 2.61e+02]
& 1.1000e+00 & 1.4519e-02 \\ \hline

Direct floor 
& [0, 12, 260]
& 8.6000e-01 & $\infty$ \\ \hline

Modified 
& [1, 1, 29]
& 1.0790e+00 & 1.3503e-01 \\ \hline

Direct ceil
& [1, 13, 261]
& 1.9110e+00 & 1.1997e-02 \\ \hline

Iterative ceil
& [1, 11, 199]
& 1.7490e+00 & 1.4514e-02 \\ \hline

\multicolumn{4}{|c|}{\textbf{Tolerance: $\epsilon^2/\sigma_1^2 = 5 \times 10^{-5}$}} \\ \hline

Real-valued 
& [6.19e+01,   3.64e+03,   7.57e+04]
& 3.1941e+02 & 5.0000e-05 \\ \hline

Direct floor 
& [61,        3637,       75650]
& 3.1850e+02 & 5.0145e-05\\ \hline

Direct ceil
& [62,        3638,       75651]
& 3.1955e+02 & 4.9978e-05\\ \hline

Iterative ceil
& [62,        3636,       75612]
& 3.1941e+02& 5.0000e-05\\ \hline

\end{tabular}
\end{minipage}

\end{tabular}
}
\caption{Model parameters and MFMC sample allocations for the Ishigami benchmark \cite{QiPeOMVe:2018}. Continuous and integer-valued MFMC allocations are compared for two prescribed relative variance tolerances, $\epsilon^2/\sigma_1^2 = 1.4519\times 10^{-2}$ and $5\times 10^{-5}$. All three candidate models are selected.}
\label{tab:MFMC_integer_comparison_Eg3}
\end{table}

\begin{table}[ht]
\centering
\scalebox{0.90}{
\begin{tabular}{cc} 

\begin{minipage}{0.27\linewidth}
\centering
\small
\begin{tabular}{|c|c|}
\hline
\textbf{Model} & \textbf{$\rho_{1,k}$} \\
\hline
1 & 1      \\
2 & 9.9838e-01 \\
3 & 9.9245e-01 \\
4 & 9.6560e-01 \\
5 & 7.0267e-01 \\
\hline
& $C_k$\\
\hline
1 & 1.000 \\
2 & 0.147 \\
3 & 0.026 \\
4 & 0.009 \\
5 & 0.002 \\
\hline
\end{tabular}
\end{minipage}
&
\begin{minipage}{0.73\linewidth}
\centering
\small
\begin{tabular}{|c|p{5.4cm}|c|c|}
\hline
\textbf{Method} & \textbf{Sample size} & \textbf{Cost} & \textbf{$\mathcal{V}^{\mathrm{MF}}_n$} \\ \hline

\multicolumn{4}{|c|}{\textbf{Tolerance: $\epsilon^2/\sigma_1^2 = 4.2991\times 10^{-2}$}} \\ \hline

Real-valued 
& [3.01e-01,   1.50e+00,   7.51e+00,   5.38e+01]
& 1.2000e+00 & 4.2991e-02
 \\ \hline

Direct floor 
& [1,     1,     1,     2]
& 1.1910e+00 & 5.3381e-01\\ \hline


Direct ceil
& [1,     2,     8,    54]
& 1.9880e+00 & 3.2978e-02\\ \hline


Iterative ceil
& [1,     2,     6,    38]
& 1.7920e+00 &4.2439e-02\\ \hline

\multicolumn{4}{|c|}{\textbf{Tolerance: $\epsilon^2/\sigma_1^2 = 5 \times 10^{-3}$}} \\ \hline

Real-valued 
& [2.58e+00,   1.29e+01,   6.46e+01,   4.62e+02]
& 1.0318e+01 &5.0000e-03 \\ \hline

Direct floor 
& [2,    12,    64,   462]
&9.5860e+00 &5.4421e-03\\ \hline


Direct ceil
& [3,    13,    65,   463]
& 1.0768e+01 & 4.8099e-03\\ \hline

Iterative ceil
& [3,    13,    61,   434]
& 1.0403e+01& 4.9975e-03\\ \hline

\end{tabular}
\end{minipage}

\end{tabular}
}
\caption{Model parameters and MFMC sample allocations for the heterogeneous
elasticity example \cite{GoGeElJa:2020}. Continuous and integer-valued MFMC allocations are compared for two prescribed relative variance tolerances, $\epsilon^2/\sigma_1^2 = 4.2991\times 10^{-2}$ and $5\times 10^{-3}$. The first four candidate models are selected.}
\label{tab:MFMC_integer_comparison_Eg4}
\end{table}

Tables~\ref{tab:MFMC_integer_comparison_Eg2}--\ref{tab:MFMC_integer_comparison_Eg4}, together with Figure~\ref{fig:Eg2}, confirm that the trends observed for the plasma equilibrium problem are consistent across the remaining benchmark problems. Despite the substantial differences in model costs, correlation structures, and numbers of fidelity levels, the proposed recursive allocation consistently provides the closest integer approximation to the continuous MFMC optimum.

Across all test cases, direct flooring and direct ceiling exhibit complementary shortcomings. Direct flooring may become infeasible or, when feasible, fail to make full use of the available computational budget. Direct ceiling, on the other hand, guaranties feasibility by allocating more samples than necessary, thus incurring unnecessary computational cost. By incorporating rounding decisions into the recursive residual-variance update, the proposed strategy preserves the robustness of ceiling rounding while substantially reducing its conservatism, yielding allocations that satisfy the prescribed variance tolerance with lower computational cost.

The improvement is most pronounced in the pre-asymptotic regime, where integer rounding has the greatest influence on the optimal allocation. As the prescribed tolerance decreases and the continuous optimal sample sizes become larger, the relative impact of rounding diminishes, so the differences among feasible integer-valued strategies become smaller. Even in this regime, however, the recursive allocation remains uniformly closer to the continuous optimum than direct ceiling rounding.

\section{Conclusion}
This paper presents a recursive framework for integer-valued sample allocation in multifidelity Monte Carlo methods. By reformulating the continuous MFMC allocation problem as a variance-constrained cost minimization problem, we reveal a sequential structure that enables a recursive allocation strategy based on Bellman's principle of optimality. The proposed method incorporates rounding directly into the recursive update of the residual variance budget, yielding integer-valued sample allocations that satisfy the prescribed variance tolerance at lower computational cost than direct ceiling rounding. Theoretical analysis and numerical experiments demonstrate that the proposed strategy remains significantly closer to the continuous optimum, particularly in the pre-asymptotic regime where rounding effects are most pronounced. The same recursive construction also extends naturally to multilevel Monte Carlo methods.

\section{Appendix}\label{sec:Appendix}

\subsection{Proof of Theorem \ref{thm:Sample_size_est}}

\begin{proof}

The proof proceeds in two steps. We first characterize the optimal solution for an arbitrary block partition of the sample allocation via the KKT conditions. We then compare all admissible block partitions and show that, under assumption~(ii), the singleton partition is globally optimal.

To describe all feasible sample allocations satisfying $N_1\le N_2\le\cdots\le N_K$, we group consecutive models with identical sample sizes into contiguous
blocks. Let
$P=(B_1,\ldots,B_q)$ be a partition of the fidelity indices
$\{1,\ldots,K\}$ into $q$ contiguous blocks, where
\[
B_i=\{\ell_i,\ldots,\ell_{i+1}-1\},
\qquad
1=\ell_1<\cdots<\ell_q<\ell_{q+1}=K+1.
\]
Each block $B_i$ consists of consecutive fidelity indices with a common sample size, and the sample sizes increase strictly across consecutive blocks
\begin{equation}
\label{eq:sample_size_block}
N_{\ell_i}
=
N_{\ell_i+1}
=
\cdots
=
N_{\ell_{i+1}-1},
\qquad
N_{\ell_i}<N_{\ell_{i+1}},
\qquad
i=1,\ldots,q-1.
\end{equation}

For a fixed partition $P$, the associated optimization problem is solved via the KKT conditions. Introducing the Lagrange multiplier $\lambda_0$ for the variance constraint and multipliers $\lambda_{\ell_i}$ associated with the monotonicity constraints, the Lagrangian is
%
\begin{equation*}
\mathcal{L}= \sum_{i=1}^q N_{\ell_i}\sum_{k\in B_i} C_k +\lambda_0 \left(\frac{\sigma_1^2}{N_{\ell_1}} + \sum_{i=2}^q \left(\frac{1}{N_{\ell_i-1}} - \frac{1}{N_{\ell_i}}\right)G_{\ell_i}- \epsilon^2\right)-\lambda_{\ell_1} N_{\ell_1}+\sum_{i=2}^q\lambda_{\ell_{i}}(N_{\ell_{i-1}} - N_{\ell_{i}}),
\end{equation*}
where $G_k=\alpha_k^2\sigma_k^2-2\alpha_k\rho_{1,k}\sigma_1\sigma_k$. The KKT conditions consist of
\[
\begin{array}{lll}
\left[\text{Stationarity}\right]&\partial \mathcal{L}/\partial \alpha_{\ell_i}=0,&\partial \mathcal{L}/{\partial N_{\ell_i}}=0,\quad i=1\ldots,q,\\
\left[\text{Primal feasibility}\right]&\mathcal{V}^{\text{MF}}- \epsilon^2 = 0, \\ 
\left[\text{Primal feasibility}\right] &-N_{\ell_1}\le 0,&N_{\ell_{i-1}}-N_{\ell_i} \le 0, \quad i=2\ldots,q,\\ 
\left[\text{Dual feasibility}\right]  &\lambda_{\ell_i} \ge 0,\quad i=1\ldots,q, \\ 
\left[\text{Complementary slackness}\right]  &\lambda_{\ell_1} N_{\ell_1}=0,&\lambda_{\ell_i}(N_{\ell_{i-1}}-N_{\ell_i})=0,\quad i=2\ldots,q.
\end{array}
\]
Since the variables $\alpha_k$ appear only through the quadratic terms $G_k$ in
\eqref{eq:MFMC_variance}, the optimal coefficients can be determined independently of the sample allocation. The stationarity condition gives
\[
\alpha_{\ell_i}^*
= \frac{\rho_{1,\ell_i}\sigma_1}{\sigma_{\ell_i}},
\qquad i=1,\ldots,q,
\]
with $\alpha_{\ell_1}^* = 1$. 
Only the block-start indices
$k=\ell_i$
appear in the MFMC estimator; all correction terms associated with
$k=\ell_i+1,\ldots,\ell_{i+1}-1$
vanish because the corresponding sample sizes are identical. Using $N_{\ell_i-1}=N_{\ell_{i-1}}$ for $i\ge2$ from \eqref{eq:sample_size_block} and substituting the optimal coefficients $\alpha_{\ell_i}^*$ into the variance in \eqref{eq:MFMC_variance} gives
\begin{equation}\label{eq:MFMC_var_convex}
    \mathcal{V}^{\mathrm{MF}} = \frac{\sigma_1^2}{N_1}+\sum_{i=2}^q
\left(\frac{1}{N_{\ell_{i}}}-\frac{1}{N_{\ell_{i}-1}}\right)\rho_{1,\ell_i}^2\sigma_1^2
=\sigma_1^2\sum_{i=1}^{q} \frac{\Delta_{\ell_i}}{N_{\ell_i}},
\end{equation}
where $\Delta_{\ell_i} = \rho_{1,\ell_i}^2-\rho_{1,\ell_{i+1}}^2$ for $i = 1, \dots, q$ with $\rho_{1,\ell_{q+1}} = \rho_{1,K+1} = 0$. The optimization problem \eqref{eq:Optimization_pb_sample_size} is then reduced to a problem involving only sample allocation with the block start indices
\begin{equation}\label{eq:Optimization_pb_sample_size_reduced}
    \begin{array}{ll}
    \displaystyle\min_{(N_{\ell_1},\ldots, N_{\ell_q}) \in \mathbb{R}^{q}} &\displaystyle \sum_{i=1}^q N_{\ell_i} \sum_{k\in B_i} C_k,\\
       \text{subject to} &\displaystyle \sum_{i=1}^q\frac{\Delta_{\ell_i}}{N_{\ell_i}}=\frac{\epsilon^2}{\sigma_1^2},\\[2pt]
       &\displaystyle N_{\ell_1}\ge 0,\quad \displaystyle N_{\ell_i-1}\le N_{\ell_i}, \;\; i=2\ldots,q.
    \end{array}
\end{equation}
Introducing the change of variables
\[
y_{\ell_i}=1/N_{\ell_i},
\]
transforms \eqref{eq:Optimization_pb_sample_size_reduced} into a convex optimization problem \cite{AgVeDiBo:2018}.


Stationarity of the Lagrangian with respect to $N_{\ell_i}$ yields
\[
\frac{\partial \mathcal{L}}{\partial N_{\ell_i}} =\sum_{k\in B_i}C_{k} -  \lambda_0\frac{\sigma_1^2\Delta_{\ell_i}}{N_{\ell_i}^2}-\lambda_{\ell_{i}}+\lambda_{\ell_{i+1}}=0,\quad i = 1, \ldots,q-1.
\]

Since the sample sizes are strictly increasing between consecutive blocks, complementary slackness implies $\lambda_{\ell_i} = 0$ for $i=1,\ldots, q$. The sample sizes are then
\begin{equation}\label{eq:sample_size_1}
    N_{\ell_i} = \sigma_1\sqrt{\lambda_0} \sqrt{\frac{\Delta_{\ell_i}}{\sum_{k\in B_i} C_{k}}}, \;\text{ for }\; i=1,\ldots,q,
\end{equation}
%
Substituting \eqref{eq:sample_size_1} into the variance constraint $\mathbb{V}[A^{\text{MF}}] = \epsilon^2$ yields
\[
\sqrt{\lambda_0}=\frac{\sigma_1}{\epsilon^2} \sum_{i=1}^{q} \sqrt{\Delta_{\ell_i}\sum_{k\in B_i} C_{k}}.
\]
Hence,
\[
N_{\ell_i}^* = \frac{\sigma_1^2}{\epsilon^2}\sqrt{\frac{\Delta_{\ell_i}}{\sum_{k\in B_i} C_{k}}}  \sum_{j=1}^{q} \sqrt{\Delta_{\ell_j}\sum_{k\in B_j} C_{k}} \;\text{ for }\; i=1,\ldots,q.
\]

Since the transformed optimization problem is convex, the KKT conditions are
necessary and sufficient for optimality. Therefore, the sample allocation derived above is the unique optimal allocation corresponding to the partition $P$.

Let $\mathscr P_K$ denote the collection of all contiguous block partitions
of $\{1,\ldots,K\}$. For each partition
$P\in\mathscr P_K$, let
$\mathcal W_P^{\mathrm{MF}}$
denote the minimum sampling cost obtained under the partition $P$. Evaluating the objective at the optimal block allocation gives
\begin{equation}
\label{eq:block_cost}
\mathcal{W}_{P}^{\text{MF}} = \sum_{i=1}^q \sum_{k\in B_i} C_k N_k = \sum_{i=1}^q N_{\ell_i}\sum_{k\in B_i} C_k =\frac{\sigma_1^2}{\epsilon^2}\left(\sum_{i=1}^{q} \sqrt{\Delta_{\ell_i}\sum_{k\in B_i} C_{k}}\right)^2.
\end{equation}

We now compare these optimal costs over all different partitions. For an arbitrary partition $P$,
the summation can be regrouped according to the blocks because the blocks
form a partition of the index set
\[
\sum_{k=1}^{K}\sqrt{\Delta_kC_k}
=
\sum_{i=1}^{q}
\sum_{k\in B_i}
\sqrt{\Delta_kC_k}.
\]

Applying the Cauchy--Schwarz inequality within each block yields
\begin{equation}
    \label{eq:Global_Optimality}
    \sum_{i=1}^{q}
\sum_{k\in B_i}
\sqrt{\Delta_kC_k}
\le
\sum_{i=1}^{q}
\sqrt{
\left(\sum_{k\in B_i}\Delta_k\right)
\left(\sum_{k\in B_i}C_k\right)
}=
\sum_{i=1}^{q}
\sqrt{
\Delta_{\ell_i}
\sum_{k\in B_i}C_k
},
\end{equation}
where the last equality follows from 
$
\sum_{k\in B_i}\Delta_k
=
\rho_{1,\ell_i}^2-\rho_{1,\ell_{i+1}}^2
=
\Delta_{\ell_i}
$. 

The equality in \eqref{eq:Global_Optimality} holds if and only if either every block contains exactly one model, or $\Delta_k/C_k$ is constant within each block. The assumption~(ii) of the theorem excludes the latter possibility for every block containing more
than one model. Therefore, every non-singleton partition has a strictly larger optimal
sampling cost than the singleton partition. Let
\[
P_{\rm s}
=
(\{1\},\ldots,\{K\})
\]
be the singleton partition. Then
\[
\mathcal W_{P_{\rm s}}^{\mathrm{MF}}
<
\mathcal W_{P}^{\mathrm{MF}},
\qquad
\forall P\neq P_{\rm s},\;\;P \in\mathscr P_K.
\]
The above argument also admits a natural interpretation from the
estimator viewpoint. If $\Delta_{k-1}/C_{k-1}=\Delta_k/C_k$ within a block, then the optimal allocation satisfies $N_{k-1}=N_k$. Consequently, $
\overline A_{k,N_k}-\overline A_{k,N_{k-1}}=0,
$ so the corresponding correction term vanishes identically. Such a model therefore incurs computational cost without contributing any additional variance reduction.

For the singleton partition, the optimal block allocation reduces to
\[
N_k^*
=
\frac{\sigma_1^2}{\epsilon^2}
\sqrt{\frac{\Delta_k}{C_k}}
\sum_{j=1}^{K}\sqrt{\Delta_j C_j},
\qquad k=1,\ldots,K.
\]
Moreover, assumption~(ii) implies

\[
\frac{N_k^*}{N_{k-1}^*}
=
\sqrt{
\frac{\Delta_k/C_k}
{\Delta_{k-1}/C_{k-1}}
}
>1 ,
\]
which satisfies the monotonicity constraints. This completes the proof.


\end{proof}

\subsection{Proof of Theorem \ref{thm:MFMC_Iterative_RealValued_Sample_Size}}

\begin{proof}

Using the residual variance recursion \eqref{eq:R_k_def} together with
\eqref{eq:MFMC_Recursive_RealValued_Sample_Size},
we obtain
\begin{equation*}\label{eq:R_k_S_k}
    R_i
    = R_{i-1} - \frac{\Delta_{i-1}}{N_{i-1}^{*,(i-1)}}
    = R_{i-1}
      \left(1 - \frac{\sqrt{C_{i-1}\Delta_{i-1}}}{S_{i-1}}\right)
    = R_{i-1}\frac{S_i}{S_{i-1}},
\end{equation*}
where the last equality follows from the recursion $S_{i-1}=\sqrt{C_{i-1}\Delta_{i-1}}+S_i$ in \eqref{eq:S_i}. This immediately yields the invariance \eqref{eq:R_k_S_k_1}. Consequently, the scaling factor $S_i/R_i$ is invariant across subproblems. Hence, for any admissible starting indices $i,j\leq k$, 
\[
N_k^{*,(i)}
=
\sqrt{\frac{\Delta_k}{C_k}}\frac{S_i}{R_i}
=
\sqrt{\frac{\Delta_k}{C_k}}\frac{S_j}{R_j}
=
N_k^{*,(j)},
\]
which establishes the cross-subproblem consistency. We may therefore omit the superscript $(i)$ in \eqref{eq:MFMC_Recursive_RealValued_Sample_Size} and denote the common value by
\begin{equation}
    \label{eq:MFMC_Recursice_RealValued_Sample_Size_1}
    N_k^* = \sqrt{\frac{\Delta_k}{C_k}}\frac{S_k}{R_k}.
\end{equation}

Finally, substituting the identity
$R_i = \frac{\epsilon^2}{\sigma_1^2 S_1} S_i$ into
\eqref{eq:MFMC_Recursice_RealValued_Sample_Size_1} gives
\[
    N_k^*
    = \frac{\sigma_1^2 S_1}{\epsilon^2}
      \sqrt{\frac{\Delta_k}{C_k}},
\]
which is precisely the closed-form MFMC allocation \eqref{eq:MFMC_RealValued_Sample_Size}.
\end{proof}

\subsection{Proof of Theorem \ref{thm:MFMC_New_IntegerValued_Variance_Cost}}

\begin{proof}

We first establish that $M_k^*\le N_k^*$ for all $k=1,\ldots,K$ by induction. For $k=1$, the recursive scheme gives
$M_1^*=N_1^*$, so the claim holds trivially. For $k\ge2$, substituting the explicit expressions for
$M_k^*$ and $N_k^*$ shows that $M_k^*\le N_k^*$ is equivalent to
\[
\sum_{j=1}^{k-1}\frac{\Delta_j}{\lceil M_j^*\rceil}
\le \frac{\epsilon^2}{\sigma_1^2}\left(1-\frac{S_k}{S_1}\right),\qquad k=2,\dots,K,
\]
where $S_k$ is defined in \eqref{eq:S_i}. It therefore suffices to prove
\[
T_{m}:=\sum_{j=1}^{m}\frac{\Delta_j}{\lceil M_j^*\rceil} \le \frac{\epsilon^2}{\sigma_1^2}\left(1-\frac{S_{m+1}}{S_1}\right),\qquad m=1,\dots,K-1,
\]
which we prove by induction on $m$. Since  \(M_1^*=\frac{\sigma_1^2}{\epsilon^2}\sqrt{\frac{\Delta_1}{C_1}}S_1\) and \(M_1^* \le \lceil M_1^*\rceil\), for the base case \(m=1\),
\[
T_1=\frac{\Delta_1}{\lceil M_1^*\rceil}
\le \frac{\Delta_1}{M_1^*}=\frac{\epsilon^2}{\sigma_1^2}\left(1-\frac{S_{2}}{S_1}\right).
\]
Suppose that, for some $m-1$ with $2\le m\le K-1$,
\[
T_{m-1}\le \frac{\epsilon^2}{\sigma_1^2}\left(1-\frac{S_{m}}{S_1}\right).
\]
Since $M_m^*=\sqrt{\frac{\Delta_m}{C_m}}{S_m}/{(\frac{\epsilon^2}{\sigma_1^2}-T_{m-1})}$, it follows that

\[
T_m = T_{m-1}+\frac{\Delta_m}{\lceil M_m^*\rceil}\le \frac{\epsilon^2}{\sigma_1^2}\left(1-\frac{S_{m}}{S_1}\right) + \frac{\Delta_m}{M_m^*}=\frac{\epsilon^2}{\sigma_1^2S_m}\left(S_{m+1}+\sqrt{C_m\Delta_m}-\frac{S_mS_{m+1}}{S_1}\right)=\frac{\epsilon^2}{\sigma_1^2}\left(1-\frac{S_{m+1}}{S_1}\right).
\]
where the inequality follows from the induction hypothesis together with $\lceil M_m^*\rceil\ge M_m^*$. This completes the induction and proves
\[
M_k^*\le N_k^*\qquad \text{for all } k=1,\dots,K.
\]
Monotonicity of the ceiling implies 
\begin{equation}\label{eq: Monotonicity_of_M_k_N_k}
    \lceil M_k^* \rceil \le \lceil N_k^* \rceil, \quad \text{ for all}\; k.
\end{equation}
Moreover, $\lceil M_k^* \rceil=\lceil N_k^* \rceil$ if and only if $\lceil N_k^*\rceil-1 < M_k^* \le N_k^*$ for all $k$.

The lower bound in
\eqref{eq:Iterative_integer_sample_size_variance_bound}
follows immediately from
\eqref{eq: Monotonicity_of_M_k_N_k},
since
$x\mapsto\Delta_k/x$
is strictly decreasing on $(0,\infty)$.
For the upper variance bound of \eqref{eq:Iterative_integer_sample_size_variance_bound}, observe that the recursive construction \eqref{eq:MFMC_New_IntegerValued_Sample_Size} satisfies
\[
\frac{\Delta_K}{M_K^*}
=
\frac{\epsilon^2}{\sigma_1^2}
-
\sum_{j=1}^{K-1}
\frac{\Delta_j}{\left\lceil M_j^* \right\rceil}.
\]
Since $\lceil M_K^* \rceil \ge M_K^*$, it follows that
\[
\sum_{k=1}^K
\frac{\Delta_k}{\left\lceil M_k^* \right\rceil}
\le
\frac{\epsilon^2}{\sigma_1^2},
\]
Equality holds if and only if $\lceil M_k^* \rceil = N_k^*$ for all $k$. Since
\(
\lceil M_k^*\rceil=N_k^*
\)
and
\(N_k^*\)
is integer-valued, substituting this relation into
\eqref{eq:MFMC_New_IntegerValued_Sample_Size}
shows that \( M_k^*=N_k^* \) for all \(k\).

The upper cost bound in \eqref{eq:Iterative_integer_sample_size_cost_bound}
follows directly from \eqref{eq: Monotonicity_of_M_k_N_k}. For the lower cost bound, applying the Cauchy--Schwarz inequality together with the upper variance bound in
\eqref{eq:Iterative_integer_sample_size_variance_bound}
yields
\[
\left(\sum_{k=1}^K \sqrt{C_k \Delta_k}\right)^2
\le
\left(\sum_{k=1}^K C_k \left\lceil M_k^* \right\rceil\right)
\left(\sum_{k=1}^K \frac{\Delta_k}{\left\lceil M_k^* \right\rceil}\right)
\le 
\frac{\epsilon^2}{\sigma_1^2}\left(\sum_{k=1}^K C_k \left\lceil M_k^* \right\rceil\right),
\]
which proves the desired result. Equality holds if and only if there exists a constant $\lambda>0$ such that
\[
\sqrt{C_k \left\lceil M_k^*\right\rceil}=\lambda\sqrt{\frac{\Delta_k}{\left\lceil M_k^*\right\rceil}},
\quad
k=1,\ldots,K,
\qquad \text{and } \quad \sum_{k=1}^K \frac{\Delta_k}{\left\lceil M_k^*\right\rceil}=\frac{\epsilon^2}{\sigma_1^2}.
\]
These conditions imply
\[
\left\lceil M_k^*\right\rceil 
= \lambda\sqrt{\frac{\Delta_k}{C_k}}
= \frac{\sigma_1^2}{\epsilon^2}\sum_{k=1}^K \sqrt{C_k\Delta_k}\sqrt{\frac{\Delta_k}{C_k}}
=N_k^*,
\]
and therefore require the continuous MFMC optimum $N_k^*$ to be integer-valued.
Together with the equality condition for the upper variance bound in
\eqref{eq:Iterative_integer_sample_size_variance_bound},
this further implies $M_k^*=\lceil M_k^*\rceil=N_k^*$ for all $k$.

\end{proof}

\bibliographystyle{abbrv}
\bibliography{references_liang}
\end{document}